\documentclass[12pt,a4paper,leqno]{amsart}
\usepackage{amssymb}
\usepackage{upref}
\usepackage[mathcal]{eucal}
\usepackage{graphicx}

\newcommand{\ma}[4]{\scriptsize{\left[ \begin{array}{cc} #1
& #2 \\ #3 & #4  \end{array} \right]}}

\begin{document}

\title[Counting Berg partitions via Sturmian words]{Counting Berg partitions via Sturmian words \\
 and substitution tilings}
\author{Artur Siemaszko and Maciej P. Wojtkowski}
\address{Department of Mathematics and Computer Science\newline University of Warmia and Mazury in Olsztyn
\newline
S\l oneczna 54 \newline 10-710 Olsztyn, POLAND}
\thanks{Research partially supported by the NCN grant 2011/03/B/ST1/04427 }
\email{artur@uwm.edu.pl, wojtkowski@matman.uwm.edu.pl}

            \theoremstyle{plain}
\newtheorem{lemma}{Lemma}
\newtheorem{proposition}[lemma]{Proposition}
\newtheorem{theorem}[lemma]{Theorem}
\newtheorem{corollary}[lemma]{Corollary}
\newtheorem{fact}[lemma]{Fact}
   \theoremstyle{definition}
\newtheorem{definition}{Definition}
    \theoremstyle{example}
\newtheorem{example}{Example}
    \theoremstyle{remark}
\newtheorem{remark}{Remark}

\newcommand{\macierz}[4]{\scriptsize{\left[\begin{array}{cc} #1 & #2 \\ #3 & #4 \end{array}\right]}}

\begin{abstract}
We develop the connection of Berg partitions with special substitution tilings
of two tiles. We obtain a new proof that the number of Berg partitions
with a fixed connectivity matrix is equal to half of the sum of its entries,
\cite{S-W}.

This approach together with the formula of S\'{e}\'{e}bold \cite{Seb},
for the number of substitutions preserving a given Sturmian sequence,
shows that all of the combinatorial substitutions can be realized
geometrically as Berg partitions.

We treat Sturmian tilings as intersection tilings of bi-partitions.  Using
the symmetries of bi-partitions we obtain geometrically the palindromic
properties of Sturmian sequences (Theorem 3)
established combinatorially by de Luca and Mignosi, \cite{L-M}.
\end{abstract}

\maketitle

\section{Introduction}\label{int}
Let us consider the interval exchange map of two intervals, $J_1$ and $J_2$,
with lengths $s_1$ and $s_2$ respectively. We denote the map by $h: J^s \to J^s,
\ J^s = J_1 \cup J_2$. The discontinuity of the map is superficial. Indeed,
when the endpoints of $J^s$ are identified, we obtain the  rotation of the
circle with the rotation number $\frac{s_1}{s_1+s_2}$,
or $\frac{s_2}{s_1+s_2}$, depending on the order of the intervals.
Let us further consider the suspension flow of $h$ with the roof function
constant on $J_1$ and $J_2$,
and equal to $u_1$ and $u_2$ respectively. We obtain the  flow
$H^t: R_1\cup R_2 \to R_1\cup R_2, t\in \mathbb R$, where
$R_1 = [0,s_1]\times[0,u_1]$ and  $R_2 = [-s_2,0]\times[0,u_2]$. Again the
discontinuity of $H^t$ is superficial.

Let $L \subset \mathbb R^2$
be the lattice generated by two translations $e_0 = (s_2,u_1)$ and
 $f_0 = (-s_1,u_2)$.
Translations of the rectangles $R_1$ and $R_2$ by vectors from the lattice
$L$ form a 2-periodic tiling of the plane.
Hence $R_1 \cup R_2$ is a fundamental domain
of the torus $\mathbb T^2 = \mathbb R^2/L$.

Partition of the torus into $\{R_1,R_2\}$ is called a \emph{bi-partition}.
The union of their horizontal sides in the torus $\mathbb T^2$
is covered by the interval of the horizontal axis $J^s = [-s_2,s_1]$,
the \emph{horizontal spine}. The union of their vertical sides
is covered by the interval of the vertical axis $J^u = [0,u_1+u_2]$,
the \emph{vertical spine}.
Let us note that $J^s\cap J^u$ contains four points in the torus, two endpoints of the
horizontal spine, and two endpoints of the vertical spine.

The flow $H^t$
is given by the constant vertical vector field in the torus $\mathbb T^2$.
Indeed for this field the return map to the horizontal spine
$ J^s = J_1 \cup J_2$ is equal to the interval exchange transformation.

Further let
$(x,y)$ be the Cartesian coordinates in $\mathbb R^2$ and
$(\xi,\eta)$ be the coordinates associated with the basis $(e_0,f_0)$, i.e.,
$ x = s_2\xi -  s_1\eta, y =  u_1\xi + u_2\eta$. We get
$\dot\xi = s_1d^{-1}, \dot\eta = s_2d^{-1}$, where $d = s_1u_1 +s_2u_2$.
In particular we observe that the slope of the trajectory in the
$(\xi,\eta)$ coordinates is equal to $s_2s_1^{-1}$ and hence it is
independent of the roof function.

For a given hyperbolic matrix $F \in GL(2,\mathbb Z)$ with non-negative
entries, let $(s_1,s_2)$ and $(u_1,u_2)$ be, respectively,
the column and the row
eigenvectors with positive entries. The automorphism $\mathcal F$
of the torus $\mathbb T^2$
defined   in the basis $(e_0,f_0)$ by $F$ preserves the horizontal and vertical
lines, i.e., it is given by a diagonal matrix in the Cartesian coordinates
$(x,y)$,  \cite{S-W}. It follows that  the bi-partition   $\{R_1,R_2\}$
is a Berg partition, i.e. it has the Markov property. Indeed in the case of $F$
with both positive eigenvalues ($\det F =1 $) we get
$\mathcal F (J^u) \supset J^u$ and
$\mathcal F(J^s) \subset J^s$. These inclusions hold because both spines
contain the origin, which is a  fixed point of $\mathcal F$.

It can be checked
that the last inclusion holds also in the case of negative stable eigenvalue,
and we will prove it later on.

Hyperbolic automorphisms have in general many
fixed points. Translating the bi-partition $\{R_1,R_2\}$ so that both spines
contain a fixed point of $\mathcal F$, we obtain nonequivalent Berg partitions
with the same connectivity matrix $F^T$,
at least in the case of both positive eigenvalues.
It was proven in \cite{S-W} that
\begin{theorem}\label{main}
The number of nonequivalent Berg partitions for
$\mathcal F$
with the connectivity matrix
$F^T$ is equal to $\left[\frac{\sigma}{2}\right]$,  where $\sigma$
is the sum of all
entries of $F$, and $[\cdot]$ denotes the integer part of a number.
\end{theorem}
The formula holds also for the toral automorphisms $\mathcal F$
defined by $-F$ in the basis $(e_0,f_0)$. The proof given in \cite{S-W}
has three distinct cases of eigenvalue signs,
with different geometry, and the common value for the number of Berg partitions
was somewhat surprising. In this paper we give another proof of Theorem~\ref{main},
based on  the connection of Berg partitions with 1-dimensional
substitution tilings, \cite{PF}.

The 1-dimensional tilings defined by the intersection of vertical lines with
a bi-partition  are associated with infinite
Sturmian sequences.  If the bi-partition is a Berg partition
then the tilings are preserved under an appropriate substitution.

The number of different substitutions preserving a given Sturmian sequence
was established by S\'{e}\'{e}bold \cite{Seb} to be equal to the sum of entries
of $F$ minus one. It is essentially the same number as the number of nonequivalent
Berg partitions from Theorem~\ref{main}. The difference arises from the fact
that
reversed substitutions come from equivalent Berg partitions.

This equality shows that all combinatorial substitutions of  Sturmian
words can be realized geometrically as  toral automorphisms acting
on respective Berg partitions. It allows the translation of combinatorial
questions regarding substitutions of Sturmian words into
the geometric language. We describe this connection in Section 5.
Let us note that although we have two essentially identical formulas
they do not follow formally from each other. A priori one can only claim
that the number of Berg partitions  does not exceed the  number of
substitutions.

In Sections 2 and 3 we develop the geometric approach to Sturmian sequences
as intersection tilings of 2-periodic tilings of the plane. Using the
symmetries
of the 2-periodic tilings we give a geometric proof of the
palindromic properties
of Sturmian sequences which were established combinatorially
by de Luca and Mignosi, \cite{L-M}. These properties are crucial in
the counting of Berg partitions obtained in Section 4.

\section{Intersection tilings}\label{til}
Given a bi-partition $\{R_1,R_2\}$, as described above,
we get the 2-periodic tiling of $\mathbb R^2$ by the rectangles
$R_1 +L$ and $R_2+L$. We will consider the \emph{intersection tilings}
of vertical lines in $\mathbb R^2$. Namely, for a given vertical line $l$,
intersecting the horizontal spine $J^s$,
we consider its partition by the intersection with the rectangles of the
2-periodic tiling. Another way to describe the intersection tiling is
to cut the line $l$ by the translates
of the horizontal spine $J^s + L$. The lengths of these intervals ({\it tiles})
are $u_1$ and $u_2$, and we will assume for convenience that $u_1 \neq u_2$.

We describe such a tiling by a sequence
$\omega \in \Omega = \{a,b\}^{\mathbb Z}$, where the letters $a$ and $b$
correspond to
the intersection with the translates of $R_1$
and $R_2$, respectively. For a sequence $\omega = \{c_n\}$ the symbol
$c_0$ corresponds to the tile with the lower endpoint on the horizontal
spine $J^s$.

The set of sequences describing intersection tilings for a given bi-partition
will be denoted by $\mathcal T \subset \Omega$. The sequences in $\mathcal T$ are
very special. Several combinatorial descriptions are available.
A particularly friendly introduction into the subject can be found
in the paper of Series \cite{Ser}. We will elaborate some of these
properties in what follows.

Let us note that the space of tilings $\mathcal T \subset \Omega$ is invariant
under the shift. It can be identified with the horizontal spine $J^s$,
and the shift becomes the interval exchange map $h:J^s \to J^s$. The vertical
lines which intersect translates of the vertical spine $J^u+L$ give rise to
two tilings differing by the exchange of $[a,b]$ and $[b,a]$. Indeed
the intersection with $J^u+z, z \in L$ can be interpreted
as the intersection with first $R_1+z$ and then $R_2+z$ (intersection on
the ``right side''), or the other way around
(intersection on the ``left side'').  In this way two sequences from $\mathcal T$
are glued into one point in $J^s$.

The study of such 1-dimensional tilings is thus equivalent to the
analysis of the symbolic dynamics for the rotation  $h$
 and the partition $J^s = J_1\cup J_2$. It goes back to the papers of
Hedlund
 and Morse \cite{H-M1}, \cite{H-M2}. Their Sturmian sequences
coincide with our intersection tilings up to countably many
nontransitive sequences.

We can further introduce the tiling space $R_1\cup R_2$, as the suspension
of the shift on $\mathcal T$. It carries the Kronecker flow $H^t$. Both of these
tiling spaces come with respective invariant Lebesgue measures.
They have their distinct advantages and we will freely use one or the other.

If the rotation number $\frac{s_1}{s_1+s_2}$ is rational then
there are only finitely many tilings, shifts of a single periodic tiling.
Indeed in such a case the lattice $L$ contains a vertical vector,
and the corresponding translation preserves the 2-periodic tiling of
$\mathbb R^2$, and
all the intersection tilings.

The structure of the periodic tiling is far from arbitrary. We will
elaborate on its structure in the following.

In the irrational case there are continuum of different tilings.
For simplicity of formulations we will assume that
$\frac{s_2}{s_1}$ is irrational, and we include some comments on the rational
case. In the irrational case two tilings are the same up to the shift
if and only if
one of the lines is a translate of the other by a vector from the lattice
$L$. Indeed if a translation takes one tiling into another then
the closure in $\mathbb T^2$ of the set of all $a$-tiles from
the first line is mapped by the translation into the
closure of the set of all $a$-tiles in the second line. These closures
are equal to $R_1$ and hence the translation vector must belong to the
lattice $L$.

\section{Symmetry of the intersection tilings}\label{til}

The group of symmetries preserving the 2-periodic tiling  of the plane,
by the translates of $R_1$ and $R_2$,  contains translations from $L$.
It also contains rotations by $\pi$ around the  centers of the vertical and
horizontal spines, and the centers of the rectangles $R_1, R_2$,
or their translations by $L$. The set of fixed points of such rotations
can be described as $(0,\frac{1}{2}(s_1-s_2))
+ \frac{1}{2} L$. Indeed, the rotation by $\pi$ around the center of
$J^s$ has four fixed points in the torus $\mathbb T^2$: the centers of
$J^s$ and $J^u$, and the centers of $R_1$ and $R_2$. Moreover for any two
copies of $J^u$ in the plane $\mathbb R^2$ the rotation by $\pi$ which
exchanges them, preserves the intersection  $J^s\cap J^u \subset
\mathbb T^2$, and hence also the horizontal spine $J^s$.
It follows that this rotation takes a copy of $J^s$ in $\mathbb R^2$
into another copy,  and thus preserves the
2-periodic tiling of $\mathbb R^2$.

The vertical line through the origin,
and its translations by vectors from $L$ have two tilings,
arising from two tilings of the vertical spine, $[a,b]$ and
$[b,a]$. These tilings will be called {\it principal},
and their ambiguous part where the vertical line intersects
a translation of the vertical spine will be called {\it the lock}.
The principal tiling  $\{c_i\} \in \mathcal T$ is
unique up to the shift and the change of
the lock described above. Further it is {\it palindromic} in the sense that
if the lock is equal to $[c_k,c_{k+1}]$ then $c_{k-r} = c_{k+1+r}$
for $r=1,2,...$. Indeed, for a line $l$ with the principal
tiling we  rotate the plane by $\pi$ around the center
of the copy of the vertical spine $J^u$ in our line $l$. The line
$l$ and the 2-periodic tiling are preserved, and the intersection tiling
is reversed, except for the lock.

More generally we say that a 1-dimensional tiling $\{c_i\} \in \Omega$
is {\it almost palindromic} if there are integers $k, l$ such that
$c_{k-r} = c_{l+r}$ for $r=1,2,...$.

\begin{proposition}\label{5pal}
If an intersection tiling is almost palindromic then it must
be the principal tiling or, up to the shift,
the tiling of one of three lines passing through
the centers of $R_1, R_2$ and $J^s$.
\end{proposition}
\begin{proof}
Consider a line $l$ with an almost palindromic tiling and the rotation by
$\pi$ around the point in $l$ which realizes the symmetry of the tiling.
In particular it exchanges  the respective points of intersections of $l$ with
translates of $J^s$. This rotation factors to the torus, and the set
of intersection points (the endpoints of tiles exchanged by the rotation)
projects to a dense subset of $J^s \subset \mathbb T^2$. It follows that
the rotation of $\mathbb T^2$ takes $J^s$ onto itself, and hence it preserves
the bi-partition. Thus the line $l$ contains the translate of one
of the four centers of $J^s, J^u, R_1,R_2$.
\end{proof}

Up to the shift there are exactly five almost palindromic tilings,
two containing $J^u$ (the lock), and three actually palindromic
through the centers of $J^s,R_1$ and $R_2$.

While there are only five ``global'' palindromes among the intersection
tilings, there is an abundance of ``local'' palindromes.
Let us consider bases $(e_i,f_i), i\geq 0$ from the fan of bases, defined
by the horizontal and vertical axes, and the lattice $L$. Let us recall
(\cite{S-W}) that the fan of bases contains all the bases $(e,f)$
of $L$ such that $e$ is in the first quadrant, and $f$ is in the second
quadrant. The fan of bases is described by the cutting sequence
$(s_i), s_i \in \{0,1\}$, and it is exhausted by the sequence of bases
$(e_i,f_i), i\in \mathbb Z$, where $e_{i+1} = e_i+f_i, f_{i+1} = f_{i}$
if $s_i = 1$, and   $e_{i+1} = e_i, f_{i+1} = e_i+f_i$
if $s_i = 0$. We have
$$
\left[
\begin{array}{cc} e_i & f_i  \end{array} \right]
=\left[
\begin{array}{cc} e_0 & f_0  \end{array} \right]
\left[
\begin{array}{cc} k_i & l_i \\ m_i & n_i  \end{array} \right],
\ \ \ \text{ where } \ \ \
\left[
\begin{array}{cc} k_i & l_i \\ m_i & n_i  \end{array} \right]
\in SL(2, \mathbb Z),
$$
and $k_i, l_i, m_i, n_i\geq 0$ for $i \geq 0$.

\begin{theorem}\label{palindromes}
For a basis $(e_i,f_i), i \geq 2$, let $p =k_i+m_i, r = l_i + n_i$,
and $p\geq 2, r\geq 2$.
In the principal tiling $\{c_j\}$ with the lock $[c_0,c_1]$
the three words
$
(c_2,\dots, c_{p-1}), (c_2,\dots, c_{r-1}),  (c_2,\dots, c_{p+r-1})
$
are palindromes.

Moreover
\[
\begin{aligned}
(c_2,\dots, c_{p+r-1})& =(c_2,\dots, c_{p-1},\ b,a,\ c_2,\dots, c_{r-1}) \\
  &=  (c_2,\dots, c_{r-1},\ a,b,\ c_2,\dots, c_{p-1}).
\end{aligned}
\]

\end{theorem}
The properties of Sturmian words from this theorem
were formulated by de Luca and Mignosi in \cite{L-M}.

For any relatively prime numbers $p$ and $r$ there
is exactly one palindromic word
$(c_2,\dots, c_{p+r-1})$ such that
$(c_2,\dots, c_{p-1})$ and $(c_2,\dots, c_{r-1})$ are also
palindromic. By exactly one we mean up to the exchange of
the symbols $a$ and $b$. This fact was formulated by Raphael M.
Robinson \cite{R} and proven, among others, by Pedersen \cite{P}.

\begin{proof}
Let us consider all the lattice points in the first quadrant, and
its convex hull. The boundary of the convex hull (the Klein's sail)
contains lattice points (i.e., points from $L$). They can be described
using the fan of bases $(e_j,f_j), j \in \mathbb Z $,
associated with the lattice
$L$ and the horizontal and vertical axes, \cite{S-W}. Namely the set of points
on the boundary of the convex hull is equal to $\{e_j| j\in \mathbb Z\}$.
It follows that there are no lattice points in the rectangle
$Q_e$ with the horizontal base, the origin at the lower left vertex,
and the diagonal equal to $e_i$,
shown in Fig. 1.
\begin{figure}[ht]{Fig. 1}
  \centering
    \includegraphics[width=1\textwidth]{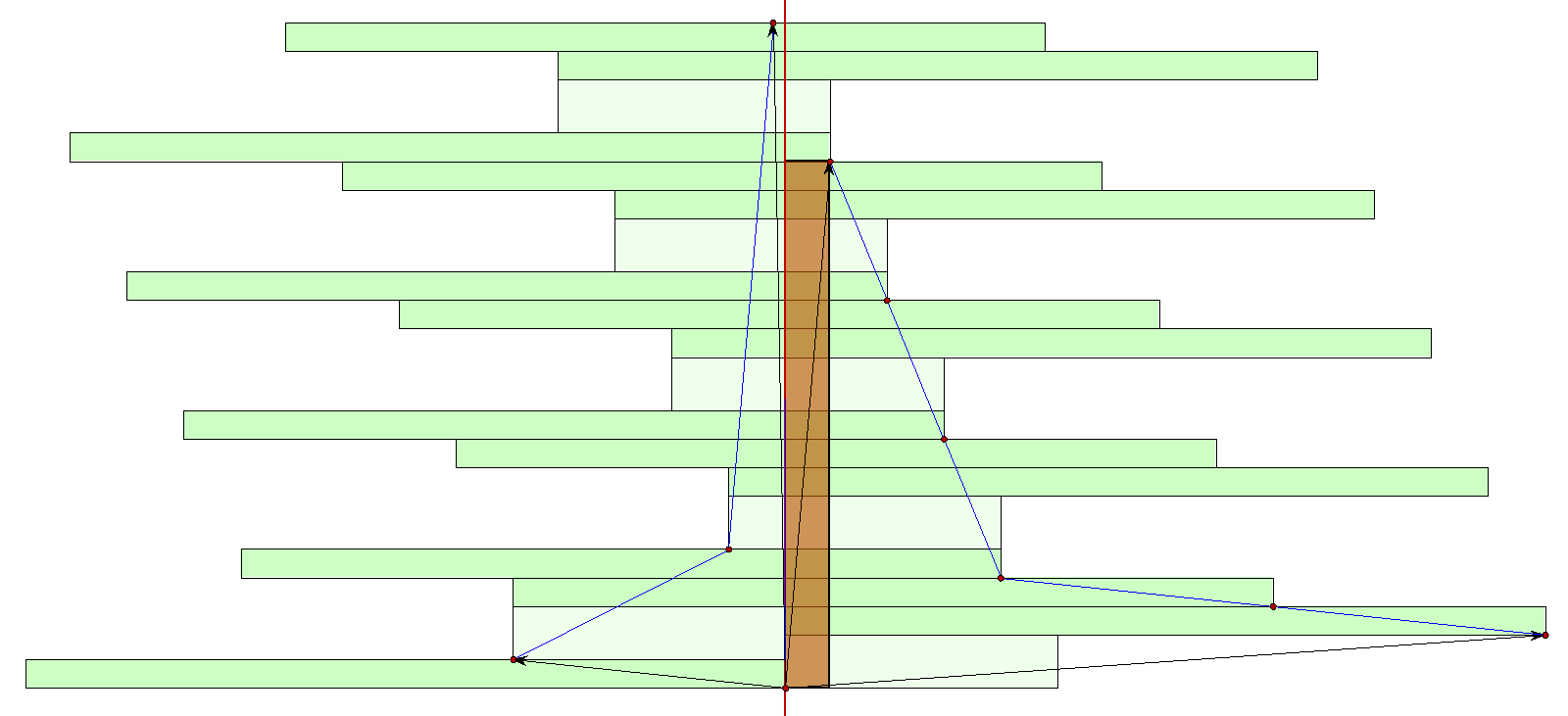}
\end{figure}
We consider the rotation by $\pi$ which exchanges $J^u$ and $J^u+e_i$.
It takes the vertical axis $l_0$ into $l_0+e_i$. As we move the line $l_0$
towards $l_0+e_i$ the part of tiling $(c_0,c_1,c_2, \dots,c_{p-1})$ does not
change, since the tilings change only at the crossing of a translated vertical
spine and there are none on the way.

Since the principal tiling is palindromic outside of the lock, it follows that
$(c_2, \dots,c_{p-1})$ must be a palindrome.

The proof for second palindrome requires the introduction of the
rectangle $Q_f$ with the horizontal base, the origin at the lower right vertex,
and the diagonal equal to $f_i$. The rest of the argument is the same.
Clearly it can also be repeated for  $e_i+f_i$ and the third palindrome.

To prove the second part let us consider the four rectangles $Q_e$, $Q_f$, $Q_e+f_i$ and $Q_f+e_i$.
They all fit into the rectangle $Q$ with the horizontal base
containing the origin, and the other sides containing $e_i$, $f_i$ and
$e_i+f_i$, respectively, Fig.2 (note that it has  different scalings
than Fig.1). Since there are no  lattice points in
the interior of $Q$,
the tilings on the left and right sides of $Q$ differ only by the change
in the lock.

\begin{figure}[ht]{Fig. 2}
  \centering
    \includegraphics[width=1\textwidth]{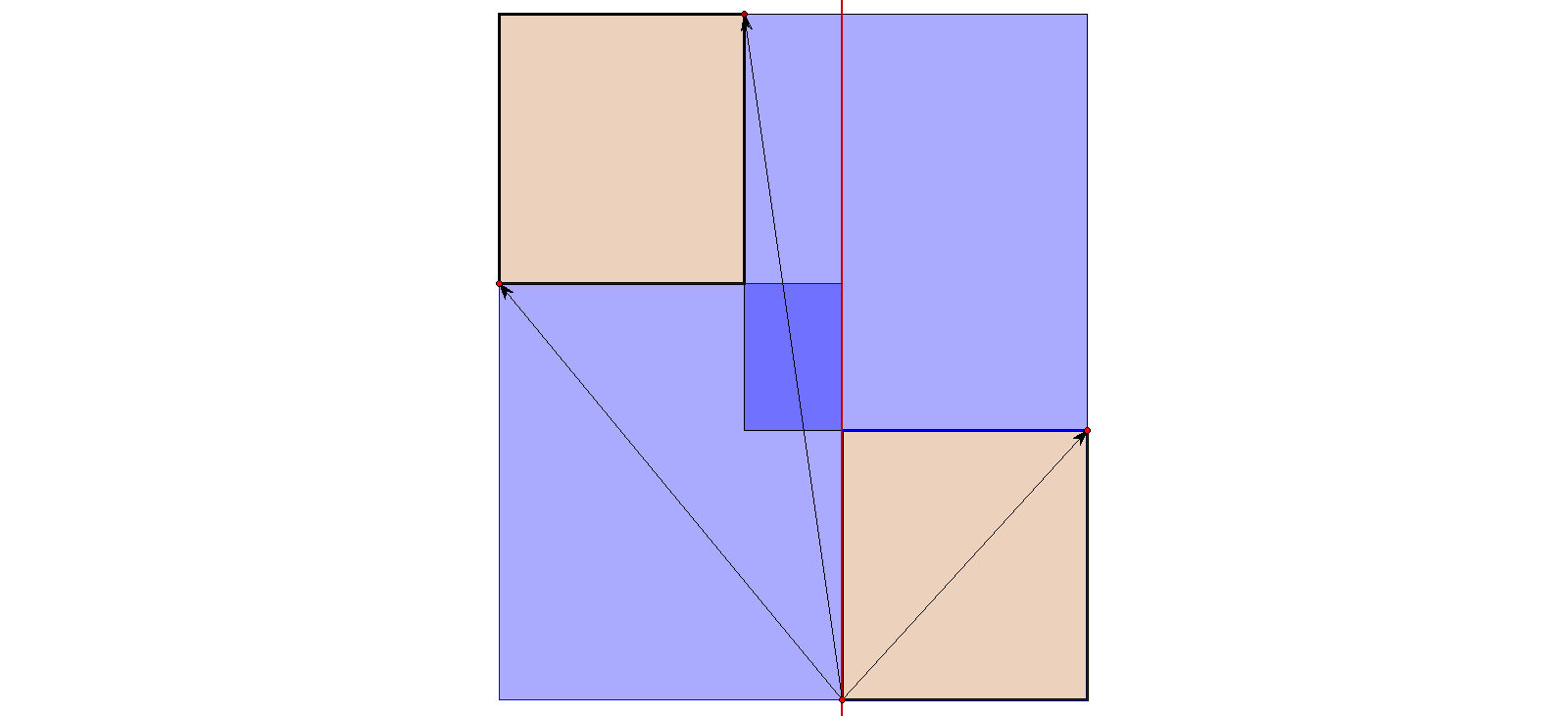}
\end{figure}

The tilings in $Q_e$ and $Q_e+f_i$,
(also in $Q_f$ and $Q_f+e_i$) are the same.
The segment $l_0+e_i\cap Q$ is covered
by $Q_e$ and $Q_f+e_i$, while the segment
$l_0+f_i\cap Q$ is covered
by $Q_f$ and $Q_e+f_i$. It gives us the two representations of
the longest palindrome

\end{proof}

Let us note that each of the two parts of
Theorem~\ref{palindromes} allows algorithmic build up of the principal
tiling based on the cutting sequence. Indeed let the one sided cutting sequence
be equal to  $(s_0,s_1,\dots) = (1,1,0,1,1,1,0,0,\dots$.
We get a sequence of vectors starting from $e_0 =(1,0), f_0=(0,1)$
which are the consecutive lattice points in the boundary of Klein's sails
in the first and the second quadrants
\[
\left[ \begin{array}{c} 1  \\  1  \end{array} \right]
\left[ \begin{array}{c} 1  \\  2  \end{array} \right]
\left[ \begin{array}{c} 1  \\  3  \end{array} \right]
\left[ \begin{array}{c} 2  \\  5  \end{array} \right]
\left[ \begin{array}{c} 3  \\  8  \end{array} \right]
\left[ \begin{array}{c} 4  \\  11  \end{array} \right]
\left[ \begin{array}{c} 5  \\  14  \end{array} \right]
\left[ \begin{array}{c} 9  \\  25  \end{array} \right]
\left[ \begin{array}{c} 13  \\  36  \end{array} \right]
\]
They give us all the information necessary to build the
palindrome with $12$ symbols $a$ and $35$ symbols $b$,
and we can do it either using the knowledge of the lengths
of the palindromes alone, or the second part of
Theorem~\ref{palindromes}.
\[
\begin{aligned}
&b\\&\\
&b \ b\\&\\
&bb \ a \ bb = b \ b a \ bb =bb \ a b \ b\\&\\
&bbab  \ b \ babb =  bb \ ab  \ bbabb = bbabb \ ba \ bb\\&\\
&bbabbb \ a \ bbbabb =  bbabbbabb \ ba \ bb =  bb \ ab \ bbabbbabb\\&\\
&bbabbbab \ b \ babbbabb =
 bbabbbabbbabb \ ba \ bb = bb \ ab \ bbabbbabbbabb\\&\\
&bbabbbabbbabbbab \ \  babbbabbbabbbabb =\\
&bbabbbabbbabb \ ba \         bbabbbabbbabbbabb =\\
&bbabbbabbbabbbabb \ ab \ bbabbbabbbabb \\&\\
&bbabbbabbbabbbabbabbbab \ b \
 babbbabbabbbabbbabbbabb =  \\
&bbabbbabbbabb \ ba \ bbabbbabb
 babbbabbabbbabbbabbbabb =  \\
 &bbabbbabbbabbbabbabbbabb
 babbbabb \ ab \ bbabbbabbbabb
\end{aligned}
\]

In the rational case the fan of bases has the last
element $(e_N,f_N)$ such that $e_N+f_N$ is vertical.
Accordingly the cutting sequence ends and the tilings
are periodic, with the period equal to the sum of
components of $e_N+f_N$. Theorem~\ref{palindromes} applies also in this case,
and we can obtain the detailed structure of the period
as shown in the example.

A segment of the principal tiling containing the vertical spine
(the lock) will be called a {\it window}. Let us consider one of the
palindromes of Theorem~\ref{palindromes} of length $k+m-2$ with
$k-1$ symbols $a$ and $m-1$ symbols $b$.
\begin{proposition}\label{types}
For the principal tiling any window with $k+m$  tiles has $k$ tiles
of type $a$ and $m$ tiles of type $b$.
\end{proposition}
\begin{proof}
We have the following structure of the symmetric section of the
principal tiling
\[
(c_{k+m-1},  \dots,   c_3,  c_2, \,[c_0, c_1,]\,c_2, c_3, \dots,
c_{k+m-1})  =
(w_1w_2 \,[c_0, c_1,]\, w_1w_2)
\]
where $(w_1w_2)$ is any split of the palindrome into two words.
Hence $(w_2 \,[c_0, c_1,]\, w_1)$ is a window with $k+m$ tiles
and it has the same content as the window $([c_0, c_1,]\, w_1w_2 )$.
\end{proof}

\section{The number of Berg partitions}\label{numb}
Recall that a bi-partition is a Berg partition for a toral
automorphism $\mathcal F$,
if $\mathcal F(J^s) \subset J^s, \mathcal F(J^u) \supset J^u$, \cite{S-W}.

Let us  assume that the bi-partition $\{R_1,R_2\}$ is a Berg partition for
an automorphism $\mathcal F$ described in the basis $\{e_0,f_0\}$
by $F \in GL(2,\mathbb Z),\ \ \
F =\left[
\begin{array}{cc} k & l \\ m & n  \end{array} \right], k,l,m,n\geq 0$,
which preserves the horizontal and vertical axes.

This is the case if $(u_1,u_2)$ and $(s_1,s_2)$ are, respectively,
the row and the  column
eigenvectors of $F$ with positive entries , \cite{S-W}.
If $\lambda > 1 $ is the unstable eigenvalue of $F$ then
$\lambda u_1 = ku_1 + mu_2, \lambda u_2 = lu_1+nu_2$.

The inclusion $\mathcal F(J^u) \supset J^u$ is obvious since
$\mathcal F(J^u)  = [0,\lambda(u_1+u_2)]$. The other inclusion
$\mathcal F(J^s) \subset J^s$ is also obvious in the case $\det F = 1$.
However if the stable eigenvalue is negative (in the case
$\det F = -1$) the origin could be too close to an endpoint
for the inclusion to hold. It can be checked by direct estimates
that the endpoints of $J^s$ are sufficiently far from the origin,
but we will establish that by a method based on  studying
the contents of windows of different sizes, which will be
crucial in the following.

Let us introduce the segments of the vertical axis
$\alpha_1 = [0,u_1], \beta_2=[u_1,u_1+u_2],
\beta_1 = [0,u_2], \alpha_2=[u_2,u_1+u_2]$. The vertical spine
is split in two ways $J^u = \alpha_1\cup\beta_2 = \beta_1\cup \alpha_2$.

We consider the windows containing the palindromes
from Theorem~\ref{palindromes} with  the number of tiles $p = k+m$, $r= l+n$
and $p+r = k+m+ l+n$, respectively
$W_\alpha = ([c_0,c_1],c_2,\dots,c_{p-1}), W_\beta =([c_0,c_1],c_2,\dots,c_{r-1}), \
W_\gamma =([c_0,c_1],c_2,\dots,c_{p+r-1})$.
The window $W_\alpha$ has $k$ symbols $a$ and $m$  symbols $b$,
so its  length is equal to $\lambda u_1$. It follows that
$\mathcal F(\alpha_1) = [0,\lambda u_1]$ is equal to the window
$W_\alpha$. Similarly $\mathcal F(\beta_1) = [0,\lambda u_2]$
is equal to the window $W_\beta$. It follows that the endpoints of
$\alpha_1$ and $\beta_1$ are mapped by $\mathcal F$ into the horizontal
spine $J^s$. (Remember that the tilings are obtained by cutting a vertical
line by the translates of $J^s$.)
Since these points are also the endpoints of $J^s$ we conclude
that $\mathcal F(J^s) \subset J^s$, and hence the bi-partition is
a Berg partition.

There are translates of $\{R_1,R_2\}$ which are also Berg partitions.
We will count the number of such translates, up to the following
equivalence. We say that two Berg partitions $\{Q_1,Q_2\}$
and $\{S_1,S_2\}$ for the toral automorphism $\mathcal F$ are equivalent
if there is a toral affine map $\mathcal U$ such that
$$
\mathcal F \circ \mathcal U = \mathcal U \circ \mathcal F,
\ \ \ \mathcal U(Q_i) = S_i, i = 1,2.$$
Note that any homeomorphism commuting with a hyperbolic toral automorphism
must be an affine map, \cite{A-P}.

\begin{theorem}\label{bergs}
The number of nonequivalent Berg partitions for $\mathcal F$
which are translates of the bi-partition $\{R_1,R_2\}$ is equal to
$\left[\frac{k+m+l+n}{2}\right]$.
\end{theorem}
\begin{proof}
Instead of translating the bi-partition we fix the rectangles
$\{R_1,$ $R_2\}$ in the affine plane and translate the origin inside
$J^u$. We get immediately $\mathcal F(J^u) \supset J^u$, and if
under such a translation the bi-partition is still a Berg
partition then $\mathcal F(J^u)$ is the window $\widehat W_\gamma$ of
the principal tiling with $p+r$ tiles, which by Proposition~\ref{types} has the same
contents as  the window $W_\gamma$.

Conversely let us consider a window $W$ with $p+r$ tiles,
e.g., $W_\gamma$ is such a window.
There are $p+r-1$ such windows, and by Proposition~\ref{types}
they have the same contents, i.e., they contain $k+l$
tiles of type $a$ and $m+n$ tiles of type $b$.
It follows that their lengths are the same and equal to
$\lambda(u_1+u_2)$.  For a given window $W$ we place the origin in
$J^u$ in such a way that $\mathcal F(J^u) = W$. We claim that in this way
we will obtain a Berg partition, that is we need to prove that
$\mathcal F(J^s) \subset J^s$. By the construction the endpoints of
$J^u$ belong to $J^s$ and are mapped to the endpoints of the window $W$.
Since a tiling is a partition of a vertical line by the translates of
$J^s$ we conclude that the endpoints of $J^u$ are mapped into
$J^s$. We need to explore the image of the endpoints of $J^s$.
These two points split the lock $J^u$ in two ways,
$\alpha_1\cup\beta_2 = \beta_1\cup \alpha_2$, or in symbolic representation
$[a,b]$ and $[b,a]$.

Accordingly let us examine the split of the window $W$ into two
segments $W =w_1\cup w_2$, containing $p$ and $r$ tiles respectively.
One of the segments will be a window itself, with the exception in the
case when the common endpoint of $w_1$ and $w_2$ splits the lock.

If say $w_1$ is a window then by Proposition~\ref{types} it has the same contents as
$W_\alpha$, i.e., $k$ symbols $a$ and $m$ symbols $b$. It follows that
the segment $w_2$ (although it is not a window itself) has the same
contents as $W_\beta$. If the common point of $w_1$ and $w_2$ falls into
the lock $J^u$, then we can exploit the ambiguity in the splitting
of the lock, and choose one so that $w_1$ and $w_2$ have the same
contents as $W_\alpha$ and $W_\beta$. We conclude that the length
of $w_i$ is $\lambda u_i, i=1,2$, and $\mathcal F(\alpha_1) = w_1,
\mathcal F(\beta_2) = w_2$. Hence the endpoint $(0,u_1)$ of $J^s$
(in $\mathbb T^2$) is mapped
into the endpoint of $w_1$, which belongs to $J^s$.

We can repeat this argument splitting the window $W$ into two segments
$w_1$ and $w_2$ with reversed contents, $r$ tiles and  $p$ tiles, respectively.
We will conclude then that also the other endpoint $(0,u_2)$ of $J^s$
is mapped into $J^s$. It follows that  $\mathcal F(J^s) \subset J^s$
and our bi-partition is a Berg partition.

For different choices of the window $W$ we obtain different locations
of the origin in $J^u$. Two of thus obtained Berg partitions
are equivalent if and only if the respective locations of the origin
in $J^u$ are symmetric. Indeed, any toral affine map $\mathcal U$
preserving the bi-partition $\{R_1,R_2\}$ must map $J^u$ onto itself,
and the fixed point contained there into the other fixed point.
It happens only for $\mathcal U$ equal to the rotation by $\pi$
around the center
of $J^u$, and the two fixed points located symmetrically inside $J^u$.

We conclude that the number of nonequivalent Berg partition is
equal to $\frac{p+r-1}{2}$ if $p+r$ is odd, and $\frac{p+r}{2}$ if
$p+r$ is even.
\end{proof}
We can visualize the different Berg partitions by considering
the refinement of the 2-periodic tiling of $\mathbb R^2$ and its preimage
under $\mathcal F$. In other words we consider the partition of
the torus $\mathbb T^2$ by $J^u$ and $\mathcal F^{-1}(J^s)$.
This partition is mapped onto the partition of the torus by
$J^s$ and $\mathcal F(J^u)$. These partitions are shown in Fig.~3
for the matrix $F =\ma{5}{2}{7}{3}$, and the origin at the lower endpoint of
$J^u$. Let us note
that they are generating Markov partitions for the toral automorphism.
\begin{figure}[ht]{Fig. 3}
  \centering
    \includegraphics[width=1\textwidth]{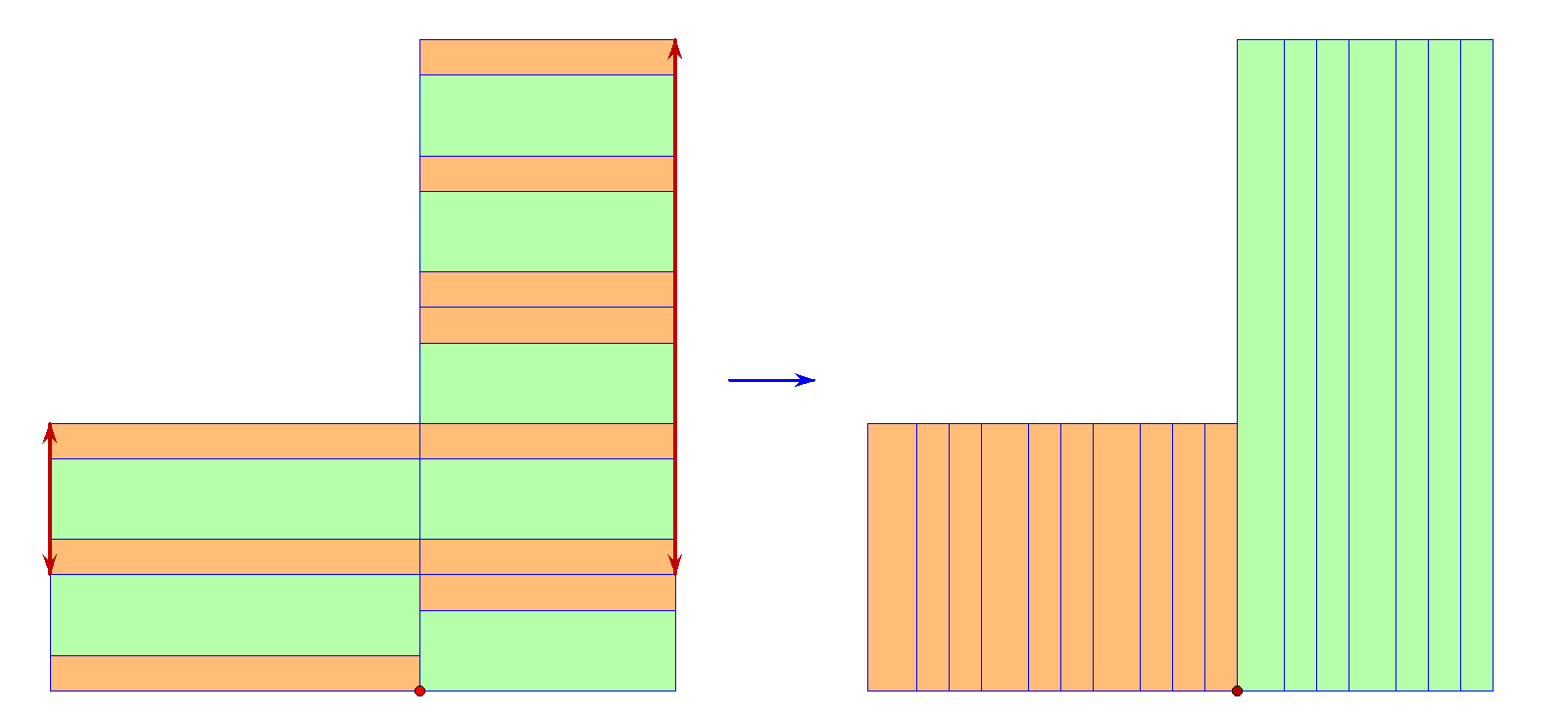}
\end{figure}
The number of wide and narrow horizontal rectangles which split
$R_1$ and $R_2$ are equal to the respective entries of the matrix $F$,
and is the same for all Berg partitions. However their order is different
for different Berg partitions. The order is the same as the order of
tiles in the two segments $w_1\cup w_2 = \mathcal F(J^u)$ from the
proof of Theorem~\ref{bergs}, with $p= k+m$ and $r=l+n$ tiles respectively.
The ambiguity in the lock is resolved by the following rule.
For $\det F = 1$ the lock is chosen as
$[a,b]$, and for $\det F = -1$ it  is chosen as $[b,a]$.

We can see that the connectivity matrix  for a Berg partition
is equal to $F^T$. Theorem~\ref{main}, proven  in \cite{S-W}, is somewhat
more general than Theorem~\ref{bergs}. It claims  that
for any toral automorphism  the number
of nonequivalent Berg partitions with a given connectivity matrix
is equal to half the sum of its entries.
It can be derived  from Theorem~\ref{bergs} as follows.

It was proven in \cite{S-W} that   a hyperbolic toral automorphism
is given by a  matrix $F \in GL(n,\mathbb Z)$ with nonnegative
entries, or $-F$, if and only if the chosen basis of the lattice $L$
is from the fan of bases associated with the stable  and
unstable lines. If the automorphism $\mathcal F$
is given by $F$ with non-negative entries then Theorem~\ref{bergs} applies.

It remains to consider the case of the toral automorphism
which is given by $-F$. Let $\mathcal S$ be the rotation by
$\pi$ around the centers of the bi-partition $\{R_1,R_2\}$,
associated with the chosen basis $(e_0,f_0)$ of the lattice $L$.
As in the proof of Theorem~\ref{bergs}, we fix the bi-partition  and translate
the origin  to define $\mathcal F$ with the representation $F$. However $\mathcal S$ is tied to
the bi-partition, and not to the shifting origin.
For any choice of the origin in $J^u$, consider the mapping $\mathcal F \circ
\mathcal S$. It has a fixed point in $J^u$, different from the origin,
unless the origin falls into the center of $J^u$. If we place the origin
in this fixed point, the mapping    $\mathcal F \circ \mathcal S$
becomes a toral automorphism given by the matrix $-F$ in the basis
$(e_0,f_0)$, and the bi-partition is its Berg partition. Indeed
$S(J^{u}) = J^u, S(J^{s}) = J^s$, and hence
$\mathcal F \circ \mathcal S (J^u) \supset J^u$ and
$\mathcal F \circ \mathcal S (J^s) \subset J^s$.
Moreover the connectivity matrices for the  Berg partitions are the same.

In this way we put the Berg partitions of $\mathcal F$ and
$\mathcal F \circ \mathcal S$ into 1-1 correspondence.
We can conclude that the number of Berg partitions with the connectivity
matrix $F^T$, for the toral automorphisms $\mathcal F$ and
$\mathcal F \circ \mathcal S$ is the same.

\section{Substitution tilings}\label{subst}

We have seen that a Berg partition for a toral automorphism $\mathcal F$
comes with the distinctive pattern of the generating Markov partition,
with the boundaries equal to $J^u \cup \mathcal F^{-1}(J^s)$, Fig 2.
These patterns are associated with substitutions preserving
the tilings. We will describe this in detail.

The toral automorphism $\mathcal F$ acts on $\mathbb T^2$,
which for any fixed bi-partition can be interpreted as the space of
tilings. If the underlying bi-partition
is a Berg partition then on the symbolic level the action
of $\mathcal F$ is by substitutions.
More precisely let us consider the  palindromes
from Theorem~\ref{palindromes} with lengths $p-2 = k+m-2$ and $r-2= l+n -2$,
respectively
$\alpha' = (c_2,\dots,c_{p-1})$and $\beta' =(c_2,\dots,c_{r-1})$.

We consider the substitutions
\[
\begin{aligned}
a \mapsto \alpha = [a,b]\alpha',  \ \ \ \ \ & a \mapsto \alpha = [b,a]\alpha',  \\
b \mapsto \beta = [b,a]\beta',  \ \ \ \ \ & b \mapsto \beta= [a,b]\beta',
\end{aligned}
\]
the left one if $\det F =1$, and  the right one if $\det F =-1$.
We will call such a substitution the {\it standard} substitution.

The principal tiling is invariant under the standard substitution.
More generally for a tiling on the vertical line passing through a point
 $h \in J^s$ the standard substitution  produces the tiling on the line
through the point $\mathcal F(h)\in J^s$. Further for any vertical line
containing a fixed  point of $\mathcal F$ the corresponding tiling will
be invariant under the standard substitution, up to the shift. It means
that after the substitution the new tiling is the same as old,
when viewed from the appropriate place, determined by the location of
the fixed point.

Shifting the origin in $J^u$, as in the proof of Theorem~\ref{bergs}, we obtain
another Berg partition, and also another substitution. It is defined
by the splitting of the window
$W = \mathcal F(J^u) = w_1\cup w_2$, where $w_1$ has $p = k+m$ tiles and $w_2$
has $r= l+n$ tiles.  The substitution is
$a \mapsto w_1,  \ b \mapsto w_2$, with the provision that
for $\det F = 1$, the lock is chosen as
$[a,b]$, and for $\det F = -1$ it  is chosen as $[b,a]$.

There are at most $p+r-1$ of these substitutions, according to the choice
of the window $W$. All of them must be actually different. Indeed we obtain
them from the standard substitution by cyclic permutations of the words
$\alpha$ and $\beta$. Should two different cyclic permutations of
$\alpha$ produce the same word, then the number of symbols $a$ (equal to $k$)
and the number of symbols $b$ (equal to $m$) would have to have a common
factor, which is not the case.  Since also $p$ and $r$ are relatively prime
there is no common multiple of $p$ and $r$ smaller than $p+r$. We conclude
that all of the $p+r-1$ cyclic permutations of $\alpha$ and $\beta$ deliver
different pairs of words, and hence different substitutions.

The counting of Berg partitions is different from the counting of substitutions
since the toral affine map $\mathcal S$ defined above conjugates
two Berg partitions with reversed substitutions.

S\'{e}\'{e}bold \cite{Seb} proved that there are exactly $k+l+m+n-1$
combinatorial substitutions
which preserve Sturmian words associated with a given matrix
$F \in GL(2,\mathbb Z)$ with nonnegative entries.

Combining this result with Theorem~\ref{bergs} we conclude that
all these combinatorial substitutions can be realized
geometrically as Berg partitions with the connectivity matrix $F^T$.
It follows that any question about the properties of Sturmian sequences
with substitution
symmetry becomes a question  about Berg partitions for the respective
toral automorphism.

\end{document}